\documentclass[12pt]{amsart}
\usepackage{amsmath,amssymb,amsthm,graphics,amscd,mathrsfs}
\usepackage{amsfonts, amssymb, amsthm}
\usepackage{fullpage, verbatim}
\usepackage[colorlinks, breaklinks, linkcolor=blue]{hyperref} 
\usepackage{breakurl}
\usepackage{array}
\usepackage{latexsym}
\usepackage{enumerate}
\usepackage{parskip}
\usepackage{tikz-cd}
\usepackage{graphicx, subcaption}

\usepackage{todonotes}

\usepackage{etoolbox}
\makeatletter
\patchcmd\deferred@thm@head
{\addvspace{-\parskip}}
{}
{}{\typeout{\string\deferred@thm@head patch failed!}}
\makeatletter

\newcommand\CA{{\mathscr A}} 
\newcommand\CB{{\mathscr B}}
\newcommand\CC{{\mathscr C}} 

\newcommand\CK{{\mathscr K}}

\newcommand\CW{{\mathscr W}}

\newcommand\Gc{{\mathcal G}}
\DeclareMathOperator{\dist}{dist}

\newcommand\BBC{{\mathbb C}}
\newcommand\BBK{{\mathbb K}}
\newcommand\BBN{{\mathbb N}}

\newcommand\BBR{{\mathbb R}}
\newcommand\BBZ{{\mathbb Z}}

\newcommand{\KK }{\mathbb{K}}

\newcommand\wt[1]{\widetilde{#1}}

\DeclareMathOperator{\coker}{coker}
\DeclareMathOperator{\im }{im}

\newcommand\Ann{{\operatorname{Ann}}}
\newcommand\codim{\operatorname{codim}}

\newcommand\Der{{\operatorname{Der}}}

\newcommand\lc{{\operatorname{lc}}}

\newcommand\Poin{{\operatorname{Poin}}}

\newcommand\res{\operatorname{res}}

\newcommand\rk{\operatorname{rk}}

\newcommand\HH{\operatorname{H}}
\renewcommand\SS{\operatorname{S}}

\numberwithin{equation}{section}

\theoremstyle{plain}
\newtheorem{lemma}[equation]{Lemma}
\newtheorem{theorem}[equation]{Theorem}

\newtheorem{corollary}[equation]{Corollary}
\newtheorem{proposition}[equation]{Proposition}
\theoremstyle{definition}
\newtheorem{defn}[equation]{Definition}
\newtheorem{remark}[equation]{Remark}
\newtheorem{remarks}[equation]{Remarks}
\newtheorem{example}[equation]{Example}


\begin{document}

\title[Formality and Combinatorial Formality]
{On Formality and Combinatorial Formality for hyperplane arrangements}

\author[T.~M\"oller]{Tilman M\"oller}
\address
{Fakult\"at f\"ur Mathematik,
	Ruhr-Universit\"at Bochum,
	D-44780 Bochum, Germany}
\email{tilman.moeller@rub.de}

\author[P.~M\"ucksch]{Paul M\"ucksch}
\address
{Mathematisches Forschungsinstitut Oberwolfach,
	Schwarzwaldstr. 9-11,
	77709 Oberwolfach-Walke,
	Germany}
\email{paul.muecksch@gmail.com}

\author[G.~R\"ohrle]{Gerhard R\"ohrle}
\address
{Fakult\"at f\"ur Mathematik,
	Ruhr-Universit\"at Bochum,
	D-44780 Bochum, Germany}
\email{gerhard.roehrle@rub.de}


\allowdisplaybreaks

\begin{abstract}
A hyperplane arrangement is called formal provided all linear dependencies among the defining forms of the hyperplanes are generated by ones corresponding to intersections  of codimension two. 
The significance of this notion stems from the fact that
complex arrangements with aspherical complements are formal.

The aim of this note is twofold.
While work of Yuzvinsky shows that formality is not combinatorial, 
in our first main theorem we prove that the combinatorial property of factoredness of arrangements does entail formality. 

Our second main theorem shows that formality is hereditary, i.e.~is passed to restrictions. This is rather counter-intuitive, as in contrast the known sufficient conditions for formality,  i.e.~asphericity, freeness and factoredness (owed to our first theorem), are not hereditary themselves.
We also demonstrate that the stronger property of $k$-formality, due to Brandt and Terao, is not hereditary.
\end{abstract}

\maketitle



\section{Introduction}

In the study of hyperplane arrangements it is of central concern to connect algebraic or topological properties (over $\BBC$)  with combinatorial invariants encoded by their intersection lattices.
Yet another important theme is to investigate the behavior of a given notion under the standard arrangement constructions of restriction and localization.
Following this philosophy, the aim of our note is to relate the combinatorial notion of factoredness with the general concept of formality on the one hand and on the other to complete the picture about the behavior of formality with respect to restriction and localization, complementing previous results in the literature.

A hyperplane arrangement 
$\CA$ is \emph{formal} provided all linear dependencies among the defining forms of the hyperplanes in $\CA$ are generated by dependencies corresponding to intersections  of codimension $2$ of the members of $\CA$ (see \S \ref{ssect:formality}). 
Over the complex numbers this property is implied by $K(\pi,1)$-arrangements (see Theorem \ref{thm:kpi1formal}). 

A property for arrangements is said to be \emph{combinatorial} if it only depends on the intersection lattice of the underlying arrangement.
Yuzvinsky \cite[Ex.~2.2]{yuzvinsky:obstruction} demonstrated that formality is not combinatorial, answering a question raised by Falk and Randell \cite{falkrandell:homotopy} in the negative.
Yuzvinsky's insight motivates the following notion.
Suppose $\CA$ is a formal arrangement. We say $\CA$ is \emph{combinatorially formal} if every arrangement with 
an intersection lattice isomorphic to the one of
$\CA$ is also formal.
In view of Yuzvinsky's result it is therefore startling that even this strong form of formality is afforded by the combinatorial property of \emph{factoredness}  (see \S \ref{ssect:factored}). This is the content of our first main result.

\begin{theorem}
	\label{thm:factored-cformal}
	If $\CA$ is factored, then it is combinatorially formal.
\end{theorem}

In \cite{falkrandell:homotopy}, the authors raised the question whether \emph{freeness} (see \S \ref{ssect:free}) 
also implies formality. This was settled affirmatively by 
Yuzvinsky in \cite[Cor.~2.5]{yuzvinsky:obstruction}.

While both, freeness and asphericity, do imply formality, it is not known whether the former are combinatorial or not (cf.~\cite[Prob.~3.8]{falkrandell:homotopyII}).
For the concept of freeness this is a longstanding conjecture due to Terao (cf.~\cite[Conj.~4.138]{orlikterao:arrangements}).

A property for arrangements is said to be \emph{hereditary} if it is inherited by every restriction of the underlying arrangement. 
Freeness, asphericity, and factoredness are known to \emph{not} be hereditary in general (cf.~\cite{edelmanreiner:orlik}, \cite[Thm.~1.1]{amendmoellerroehrle:aspherical}, and Example \ref{ex:heredfactored}, respectively).
It is thus rather surprising and counter-intuitive that formality  on the other hand, which is entailed by each of these properties, \emph{is} hereditary. This is the content of our second main theorem.

\begin{theorem}
	\label{thm:restriction-formal}
	Formality is hereditary.
\end{theorem}

The notion of formality was further refined to the concept of \emph{$k$-formality} by Brandt and Terao \cite[Def.~4.4]{brandtterao} (see \S \ref{s:complements}).
In the setting of \cite{brandtterao}, formality coincides with $2$-formality.
Higher formality is stronger than formality, cf.~\cite[Ex.~5.1]{brandtterao}. 
In \cite[Thm.~4.15]{brandtterao}, Brandt and Terao extended Yuzvinsky's result \cite[Cor.~2.5]{yuzvinsky:obstruction} by showing that 
free arrangements are $k$-formal for all $k$.
Extending \cite[Thm.~4.15]{brandtterao} further, in \cite[Thm.~1.1]{DiPasquale:freeness}, DiPasquale employed homological methods based on those developed in \cite{brandtterao} 
to characterize freeness of a multiplicity of an arrangement and showed as a consequence 
that the latter implies $k$-formality for all $k$, \cite[Cor.~4.10]{DiPasquale:freeness}.
We emphasize that Theorem \ref{thm:restriction-formal} does not extend to  the stronger notion of $k$-formality, see Example \ref{ex:3formal-not-hereditiary}.

A property for arrangements is said to be \emph{local} if it is passed to every localization of the ambient arrangement. 
Freeness, asphericity, and factoredness
are all local (cf.~\cite[Thm.~4.37]{orlikterao:arrangements}, \cite[Lem.~1.1]{paris:deligne}, and the proof of \cite[Cor.~2.11]{terao:factored},
respectively). In contrast, 
this fails for formal arrangements of rank at least $4$ (cf.~the example following \cite[Def.~2.3]{yuzvinsky:obstruction}). 
However, as an application of Theorem \ref{thm:restriction-formal}, we in turn show that 
localizations at modular flats of formal arrangements are formal again, see Corollary \ref{cor:modular-formal}.

Moreover, 
it is also known that 
if $X$ is a modular flat of $\CA$ of corank  $1$ and the localization $\CA_X$ is free, $K(\pi,1)$, or factored, then so is $\CA$ itself
(cf.~\cite[Rem.~2.8]{roehrle:ideal}, \cite{terao:modular}, and \cite[Rem.~2.25]{roehrle:ideal}, respectively).
In Corollary \ref{cor:modular-formal-corank1}, we show that this property also holds for
formality.

Following Yuzvinsky \cite[Def.~2.3]{yuzvinsky:obstruction}, we say that $\CA$ is \emph{locally formal} provided that each localization of $\CA$ is formal.
Since freeness, asphericity, and factoredness
are all local and each affords formality, each of these properties implies local formality. 
Since formality is not passed to localizations in general, as noted above, the class of formal arrangements properly encompasses the classes of free, factored and $K(\pi,1)$-arrangements.
The following is immediate from Theorem \ref{thm:restriction-formal}.

\begin{corollary}
	\label{cor:restrition-tot-formal}
	Local formality is hereditary.
\end{corollary}

The paper is organized as follows. In \S \ref{ssect:arrangements}
and \S \ref{ssect:modular}, we recall some standard terminology and introduce 
further notation 
on hyperplane arrangements and record some basic facts on modular elements
in the lattice of intersections of an arrangement.
The notions of formality and $k$-formality are 
recapitulated from \cite{brandtterao} in \S \ref{ssect:formality}. This is followed by a brief discussion of the concept of line closure of a subset of an arrangement. Here we review Falk's criterion that the presence of an lc-basis entails 
combinatorial formality (Proposition \ref{prop:lcbasis}). This in turn is the key ingredient in our proof of Theorem \ref{thm:factored-cformal}.
Short recollections on the fundamental notions of $K(\pi,1)$-arrangements (\S \ref{ssect:kpionearrangements}) and simplicial arrangements (\S \ref{ssec:SimplicialArrs}) follow.
In \S \ref{ssect:free} and \S \ref{ssect:factored}, 
the concepts of free and factored 
arrangements are recalled.

Theorems \ref{thm:factored-cformal} and \ref{thm:restriction-formal} are proved in \S \ref{sec:factored-kformal} and \S \ref{sec:restriction-formal}, respectively.

In our final section, we present some complements to our main developments. Here we  demonstrate in Example \ref{ex:3formal-not-hereditiary} that Theorem \ref{thm:restriction-formal} does not extend to higher formality.

In addition, we present a family of combinatorially formal arrangements $\CK_n$ of rank $n \ge 3$ which fail to be free, $K(\pi,1)$, and factored (Example \ref{ex:nonkpioneII}). These are natural subarrangements of the Coxeter arrangement of type $B_n$. 

For general information about arrangements  
we refer the reader to 
\cite{orlikterao:arrangements}.

\section{Recollections and Preliminaries}
\label{sect:prelims}

\subsection{Hyperplane arrangements}
\label{ssect:arrangements}
Let $\BBK$ be a field and let
$V = \BBK^\ell$ be an $\ell$-dimensional $\BBK$-vector space.
A \emph{hyperplane arrangement} $\CA = (\CA, V)$ in $V$ 
is a finite collection of hyperplanes in $V$ each 
containing the origin of $V$.
We also use the term $\ell$-arrangement for $\CA$. 
We denote the empty arrangement in $V$ by $\Phi_\ell$.

The \emph{lattice} $L(\CA)$ of $\CA$ is the set of subspaces of $V$ of
the form $H_1\cap \dotsm \cap H_i$ where $\{ H_1, \ldots, H_i\}$ is a subset
of $\CA$. 
For $X \in L(\CA)$, we have two associated arrangements, 
firstly
$\CA_X :=\{H \in \CA \mid X \subseteq H\} \subseteq \CA$,
the \emph{localization of $\CA$ at $X$}, 
and secondly, 
the \emph{restriction of $\CA$ to $X$}, $(\CA^X,X)$, where 
$\CA^X := \{ X \cap H \mid H \in \CA \setminus \CA_X\}$.
Note that $V$ belongs to $L(\CA)$
as the intersection of the empty 
collection of hyperplanes and $\CA^V = \CA$. 
The lattice $L(\CA)$ is a partially ordered set by reverse inclusion:
$X \le Y$ provided $Y \subseteq X$ for $X,Y \in L(\CA)$.

Throughout, we only consider arrangements $\CA$
such that $0 \in H$ for each $H$ in $\CA$.
These are called \emph{central}.
In that case the \emph{center} 
$T(\CA) := \cap_{H \in \CA} H$ of $\CA$ is the unique
maximal element in $L(\CA)$  with respect
to the partial order.
A \emph{rank} function on $L(\CA)$
is given by $r(X) := \codim_V(X)$.
The \emph{rank} of $\CA$ 
is defined as $r(\CA) := r(T(\CA))$.

For $1\leq k \leq r(\CA)$ we write $L_k(\CA)$ for the collection of all $X\in L(\CA)$ of rank $k$. 

\subsection{Modular elements in $L(\CA)$}
\label{ssect:modular}

We say that $X \in L(\CA)$ is \emph{modular}
provided $X + Y \in L(\CA)$ for every $Y \in L(\CA)$,
 \cite[Cor.\ 2.26]{orlikterao:arrangements}.
We require the following characterization of modular members of $L(\CA)$ of rank $r-1$,
see the proof of \cite[Thm.\ 4.3]{bjoerneredelmanziegler}.

\begin{lemma}
\label{lem:modular1}
Let $\CA$ be an arrangement of rank $r$.
Suppose that $X \in L(\CA)$ is of rank $r-1$.
Then $X$ is modular if and only if 
for any two distinct $H_1, H_2 \in \CA \setminus \CA_X$ there
is a $H_3 \in \CA_X$ so that $r(H_1 \cap H_2 \cap H_3) = 2$,
i.e.\ $H_1, H_2, H_3$ are linearly dependent.
\end{lemma}

Next, we record a refined geometric version of a standard fact
about modular elements in a 
geometric lattice, cf.~\cite[Prop.\ 2.42]{aigner:combinatorialtheory}
or \cite[Lem.\ 2.27]{orlikterao:arrangements}.

\begin{lemma}
	\label{lem:modular2}
	Let $\CA$ be an arrangement of rank $r$.
	Suppose that $X \in L(\CA)$ is modular of rank $q$.
	Then there is a complementary intersection $Y \in L(\CA)$ of rank $r-q$ 
	(i.e.\ $X\cap Y = T(\CA)$)
	such that the arrangements $(\CA_X/X,V/X)$ and $(\CA^Y/T(\CA),Y/T(\CA))$ are linearly isomorphic.
\end{lemma}
\begin{proof}
	Let $X = H_1\cap \ldots \cap H_q$. Then there are $H_{q+1},\ldots,H_r\in \CA$
	such that $T(\CA) = H_1\cap \ldots \cap H_r$, i.e.\ $Y = H_{q+1}\cap\ldots\cap H_r$ is
	complementary to $X$.
	Due to the modularity of $X$ for all $K=H'\cap Y \in \CA^Y$, we have $H = X + K \in \CA_X$,
	$H'\cap Y = H\cap Y = K$, and furthermore, this correspondence is bijective. 
	Consequently, the quotient map $Y/T(\CA) \to V/X, v + T(\CA) \mapsto v+X$, 
	which is an isomorphism by the choice of $Y$, maps $\CA^Y/T(\CA)$ to $\CA_X/X$.
\end{proof}

\subsection{Formality and $k$-formality}
\label{ssect:formality}

The notion of formality is due to Falk and Randell \cite{falkrandell:homotopy}.
Here we give the equivalent definition by Brandt and Terao \cite{brandtterao}.
Let $(e_H\mid H\in \CA)$ be the basis of a  $\BBK$-vector space indexed by the hyperplanes in $\CA$.
For each $H\in \CA$, choose a linear form $\alpha_H\in V^*$ such that $\ker \alpha_H=H$. 
Consider the map 
\[
\bigoplus \BBK e_H\rightarrow V^* \quad\text{  defined by  }\quad
e_H\mapsto \alpha_H
\]
and let $F(\CA)$ be the kernel of this map.
Note that $\CA$ is linearly isomorphic to the arrangement 
$\{F(\CA)^\perp\cap \ker x_H\mid H\in \CA\}$ in $F(\CA)^\perp$,
where $(x_H\mid H\in \CA)$ is the dual basis to $(e_H\mid H\in \CA)$. We say an element $x\in F(\CA)$ has \emph{length} $p$ if $x$ has exactly $p$ nonzero entries. Let $F_2(\CA)$ be the subspace of $F(\CA)$ generated by all $x\in F(\CA)$ of length at most $3$. Then, $\CA$ is said to be \emph{formal} if $F_2(\CA)=F(\CA)$. Note that if $r(\CA)=2$, then $F_2(\CA)=F(\CA)$, so each rank $2$ arrangement is formal. As indicated in the introduction,  
a localization of a formal arrangement need not be formal
again, see the example after \cite[Def.~2.3]{yuzvinsky:obstruction}.

The following generalization of formality is due to Brandt and Terao  \cite{brandtterao}. If $X\in L(\CA)$, then there is a natural inclusion $F_2(\CA_X)\hookrightarrow F_2(\CA)$. If $x\in F(\CA)$ is of length $3$, then there is an associated $X\in L_2(\CA)$ with $x\in F_2(\CA_X)\subset F(\CA).$ Thus, $\CA$ is formal if and only if the map
\[
\pi_2: \bigoplus\limits_{X\in L_2(\CA)} F(\CA_X)\longrightarrow F(\CA)
\]
is a surjection. The definition of \emph{$k$-formality} in \cite{brandtterao} is recursive and does not require the choice of linear forms as before. To emphasize the use of this alternative definition, we follow the notation from \cite{brandtterao}, which differs from the one above.

Let $R_0(\CA)=T(\CA)^*$ be the dual space of the center of $\CA$. For each $X\in L(\CA)$, we have a surjective map $i_0(X):R_0(\CA_X)\rightarrow R_0(\CA)$ defined by restricting the domain to $T(\CA)$. Then, for $1\leq k\leq r(\CA)$, define $R_k(\CA)$ recursively  as the kernel of the map
\[
\pi_{k-1}(\CA): \bigoplus\limits_{X\in L_{k-1}}R_{k-1}(\CA_X) \longrightarrow R_{k-1}(\CA),
\]
where the maps $\pi_k$ are defined as follows. 
(Note that $\bigoplus_{X\in L_{0}}R_{0}(\CA_X) = R_0(\Phi_\ell) = V^*$
and $\pi_0 : V^* \to R_0(\CA) = T(\CA)^*$ is just restriction.)
For each $k\geq 1$ and each $Y\in L(\CA)$ with $\rk(Y)\geq k$, there is a map
$$i_k(Y):=i_k(\CA,Y): R_k(\CA_Y) \rightarrow R_k(\CA),$$
defined recursively via the commutative diagram
\begin{equation*}
\label{cd:definitionOfPi}
\begin{CD}
R_k(\CA_Y) @>>> \bigoplus\limits_{\substack{X \in L_{k-1} \\ X\leq Y}} R_{k-1}((\CA_Y)_X) @>\pi_{k-1}(\CA_Y)>> R_{k-1}(\CA_Y) \\
@VV i_k(Y) V			@VV j_{k-1}(Y) V																			@VV i_{k-1}(Y) V	\\
R_k(\CA) @>>> \bigoplus\limits_{X \in L_{k-1}} R_{k-1}(\CA_X) @>\pi_{k-1}(\CA)>>  R_{k-1}(\CA)
\end{CD}
\end{equation*}
In the diagram, $R_k((\CA_Y)_X) = R_k(\CA_X)$ since $X\leq Y$, so $j_{k-1}(Y)$ is defined as the direct sum of identity maps and zero maps. Then $i_k(Y)$ is defined through the universal property of the kernel $R_k(\CA)$. With this define $\pi_k$ as the 
sum of the $i_k(X)$ for $X\in L_k(\CA)$. 

If $\CA$ is formal, we say it is \emph{$2$-formal}. For $k \ge 3$, we say $\CA$ is \emph{$k$-formal} if $\CA$ is $(k-1)$-formal and the map $\pi_k$ is surjective. 
Note that an arrangement of rank $r$ which is $(r-1)$-formal is trivially $k$-formal for all $k \ge r$.

\subsection{Line closure and combinatorial formality}
\label{ssect:lineclosure}

We want to establish further combinatorial properties of $\CA$ that imply formality. 
The following definitions are due to Falk \cite{falk:line-closure} for matroids,
but they can easily be applied to arrangements as well.
Let $\CB \subset \CA$ be a subset of hyperplanes.
We say $\CB$ is \emph{closed} if $\CB =\CA_Y$ for $Y=\bigcap\limits_{H\in \CB} H$.
We call $\CB $ \emph{line-closed} 
if for every pair $H,H'\in \CB $ of hyperplanes, we have $\CA_{H\cap H'}\subset \CB $. 
The \emph{line-closure} $\lc(\CB)$ of $\CB $ is defined 
as the intersection of all line-closed subsets of $\CA$ containing $\CB$. 
The arrangement $\CA$ is called \emph{line-closed} 
if every line-closed subset of $\CA$ is closed.
With these notions, we have the following criterion for combinatorial formality, see \cite[Cor.~3.8]{falk:line-closure}:
\begin{proposition}
	\label{prop:lcbasis}
		Let $\CA$ be an arrangement of rank $r$. Suppose $\CB \subseteq \CA$ consists of $r$ hyperplanes such that $r(\CB)=r$ and $\lc(\CB)=\CA$. Then $\CA$ is combinatorially formal.
\end{proposition}

We call such a subset $\CB \subseteq \CA$ as in Proposition \ref{prop:lcbasis} an \emph{lc-basis} of $\CA$. Note that the converse of Proposition \ref{prop:lcbasis} is false \cite{falk:line-closure}, i.e.~a combinatorially formal $\CA$ need not admit an lc-basis. 

Furthermore, the existence of an lc-basis does not imply higher formality. Brandt and Terao give an example of an arrangement that is $2$-formal but not $3$-formal, see \cite[Ex.~5.1]{brandtterao}. Here an lc-basis is easily calculated, e.g.~take $\{H_1,H_2,H_4,H_5\}$ in \cite[Ex.~5.1]{brandtterao}.

Finally, we note that the presence of an lc-basis does not descend to localizations, i.e.~if $\CB \subseteq \CA$ is as in Proposition \ref{prop:lcbasis} and $X \in L(\CA)$, then it need not be the case that $\CB_X$ is an lc-basis of $\CA_X$, see Example \ref{ex:3formal-not-hereditiary}.

\subsection{$K(\pi,1)$-arrangements}
\label{ssect:kpionearrangements}
A complex $\ell$-arrangement $\CA$ is called  \emph{aspherical}, or a 
\emph{$K(\pi,1)$-arrangement} (or that $\CA$ is $K(\pi,1)$ for short), provided 
the complement $M(\CA)$ of the union of the hyperplanes in 
$\CA$ in $\BBC^\ell$ is aspherical, i.e.~is a 
$K(\pi,1)$-space. That is, the universal covering space of $M(\CA)$ 
is contractible and the fundamental group
$\pi_1(M(\CA))$ of $M(\CA)$ is isomorphic to the group $\pi$.
This is an important 
topological property, for 
the cohomology ring $H^*(X, \BBZ)$ of a $K(\pi,1)$-space $X$
coincides
with the group cohomology $H^*(\pi, \BBZ)$ of $\pi$.
The crucial point here is that the intersections of codimension $2$  
determine the fundamental group $\pi_1(M(\CA))$ of $M(\CA)$. 

By Deligne's seminal result \cite{deligne}, complexified simplicial arrangements are $K(\pi, 1)$.
Likewise for complex supersolvable arrangements, 
cf.~\cite{falkrandell:fiber-type} and \cite{terao:modular}
(cf.~\cite[Prop.\ 5.12, Thm.~5.113]{orlikterao:arrangements}).
As restrictions of simplicial (resp.~supersolvable)
arrangements are again simplicial 
(resp.~supersolvable), 
the $K(\pi, 1)$-property of these kinds of arrangements
is inherited by their restrictions.
However, we emphasize that in general, a restriction of a 
$K(\pi, 1)$-arrangement need not be $K(\pi, 1)$ again, see 
\cite{amendmoellerroehrle:aspherical} for examples of this kind.

The following theorem due to Falk and Randell \cite[Thm.\ 4.2]{falkrandell:homotopy} 
establishes formality as a necessary condition for asphericity.

\begin{theorem}
	\label{thm:kpi1formal}
	If $\CA$ is a $K(\pi,1)$-arrangement, then it is formal.
\end{theorem}

Thanks to \cite{cuntzgeiss},
simpliciality is a combinatorial property.
Thus, as simplicial arrangements are $K(\pi,1)$,  
it follows that all such are combinatorially formal, see also \S \ref{ssec:SimplicialArrs}. 
It is not known whether this is true for aspherical arrangements in general. 

\begin{remark}
	\label{rem:local-kpione}
	Thanks to an observation by Oka, 
	if the complex arrangement $\CA$ is $K(\pi, 1)$, 
	then so is every localization $\CA_X$ for $X \in L(\CA)$, 
	e.g., see \cite[Lem.~1.1]{paris:deligne}.
\end{remark}

The following is an immediate consequence of Terao's work 
\cite{terao:modular} (see also \cite[\S 5.5]{orlikterao:arrangements}).

\begin{lemma}
	\label{lem:modular}
	Let $\CA$ be a complex arrangement of rank $r$.
	Suppose that $X \in L(\CA)$ is modular of rank $r-1$.
	If $\CA_X$ is $K(\pi,1)$, then so is $\CA$.
\end{lemma}

\subsection{Simplicial arrangements}
\label{ssec:SimplicialArrs}

Let $\CA$ be a real arrangement in $V \cong \BBR^\ell$. Then, the connected components of the complement $V \setminus \bigcup_{H \in \CA}H$
are called chambers and are denoted by $\CC(\CA)$.
The arrangement $\CA$ is called \emph{simplicial} if all $C \in \CC(\CA)$ are (open) simplicial cones.
It was already mentioned above that simplicial arrangements are aspherical and therefore formal.
Furthermore, simpliciality is actually a combinatorial property of the intersection lattice, as was observed in \cite{cuntzgeiss}.
Consequently, the formality of simplicial arrangements is combinatorial.

In Proposition \ref{prop:WallsLCBasis}  below, we demonstrate that a simplicial arrangement naturally satisfies the stronger property
of having an $\lc$-basis. Explicitly, owing to Proposition \ref{prop:WallsLCBasis}, 
the walls of each chamber of a simplicial arrangement yield such a special basis.

For a chamber $C \in \CC(\CA)$, we write $\CW^C := \{H \in \CA \mid \langle \overline{C}\cap H\rangle = H\}$ for the \emph{walls} of $C$.
Two chambers $C,D \in \CC(\CA)$ are \emph{adjacent} if $\langle \overline{C}\cap\overline{D}\rangle \in \CA$.

The following lemma, which is a special case of \cite[Lem.~3.3]{CM2019supersimp},
provides the essential argument for the proof of Proposition \ref{prop:WallsLCBasis} below.

\begin{lemma}
	\label{lem:SimplAdjK}
	Let $C,D \in \CC(\CA)$ be two adjacent and simplicial chambers with a common wall $H \in \CW^C\cap \CW^{D}$, and let $H' \in \CW^{D} \setminus \{H\}$.
	Then, $H' \in \CW^{C}$ if and only if $|\CA_{H\cap H'}| = 2$.
\end{lemma}

We require a bit more notation, cf.\ \cite[Sec.~2.2]{CM2019supersimp}.
A sequence $(C_0,C_1,\ldots,$ $C_{n-1},C_n)$ of distinct chambers in $\CC(\CA)$ is called a \emph{gallery} if
for all $1 \leq i \leq n$ the chambers $C_i$ and $C_{i-1}$ are adjacent. 
The set of all galleries is denoted by $\Gc(\CA)$.
The \emph{length} $l(G)$ of  a gallery $G \in \Gc(\CA)$ is one less than the number of chambers in $G$.
For a gallery $G = (C_0,\ldots,C_n)$, we denote by $b(G) = C_0$ the first chamber and by $t(G) = C_n$ the last chamber in $G$.
For two chambers $C,D \in \CC(\CA)$, we set $\dist(C,D) := \min\{l(G) \mid G \in \Gc(\CA), b(G)=C$ and $t(G)=D\}$, and 
for $H \in \CA$, we set $\dist(H,C) := \min\{l(G) \mid b(G)=C, t(G)=D$ and $H \in \CW^{D} \}$.

\begin{proposition}
	\label{prop:WallsLCBasis}
	Let $\CA$ be a simplicial arrangement. Then, for each $C \in \CC(\CA)$ we have $\lc(\CW^C) = \CA$.
\end{proposition}
\begin{proof}
	Let $H \in \CA$.
	We argue by induction on $\dist(H,C)$.
	
	Firstly, we have $\dist(H,C)=0$ if and only if $H \in \CW^C$. Hence, $H \in \lc(\CW^C)$ in this case.
	
	Now suppose $\dist(H,C)>0$.
	For the induction step, let $D \in \CC(\CA)$ be such that $\dist(H,C) = \dist(C,D)$ and $H \in \CW^{D}$.
	There is a chamber $E \in \CC(\CA)$ adjacent to $D$ with $\dist(C,E) < \dist(C,D)$.
	By the induction hypothesis, we see that $\CW^{E} \subseteq \lc(\CW^C)$.
	Now, set $H' := \langle\overline{D}\cap\overline{E}\rangle$.
	By considering the distances, we apparently have $H \in \CW^{D} \setminus \CW^{E}$ and consequently, by Lemma \ref{lem:SimplAdjK},
	there is another $H'' \in  \CW^{E} \setminus \{H'\}$ such that $H \in \CA_{H' \cap H''}$,
	i.e.\ $H \in \lc(\CW^{E}) \subseteq \lc(\CW^C)$, which concludes the induction.
	\end{proof}

In view of Proposition \ref{prop:lcbasis}, 
Proposition \ref{prop:WallsLCBasis} again shows that simplicial arrangements are combinatorially formal.

\subsection{Free hyperplane arrangements}
\label{ssect:free}
Let $S = S(V^*)$ be the symmetric algebra of the dual space $V^*$ of $V$.
Let $\Der(S)$ be the $S$-module of $\BBK$-derivations of $S$.
Since $S$ is graded, 
$\Der(S)$ is a graded $S$-module.

Let $\CA$ be an arrangement in $V$. 
For $H \in \CA$, we fix $\alpha_H \in V^*$ with
$H = \ker \alpha_H$.
The \emph{defining polynomial} $Q(\CA)$ of $\CA$ is given by 
$Q(\CA) := \prod_{H \in \CA} \alpha_H \in S$.
The \emph{module of $\CA$-derivations} is 
defined by 
\[
D(\CA) := \{\theta \in \Der(S) \mid \theta(Q(\CA)) \in Q(\CA) S\} .
\]
We say that $\CA$ is \emph{free} if 
$D(\CA)$ is a free $S$-module, cf.\ \cite[\S 4]{orlikterao:arrangements}.

\begin{remark}
	\label{rem:local-free}
Note that the class of free arrangements is closed with respect to taking 
localizations, 
cf.~\cite[Thm.\ 4.37]{orlikterao:arrangements}. 
For  $X\in L(\CA)$ modular of rank $r-1$, 
also the converse holds, e.g.~see \cite[Rem.~2.8]{roehrle:ideal}.
\end{remark}

In \cite[Cor.~2.5]{yuzvinsky:obstruction}, Yuzvinsky showed that freeness entails formality. This was extended to $k$-formality by Brandt and Terao in \cite[Thm.~4.15]{brandtterao}:

\begin{theorem}
	\label{thm:brandtterao}
	If $\CA$ is free, then it is $k$-formal for each $k$.
\end{theorem}

Note that the converse of 
Theorem \ref{thm:brandtterao} is false, e.g.~see Example \ref{ex:3formal-not-hereditiary}. 

Following DiPasquale \cite[Def.~4.9]{DiPasquale:freeness}, we say that $\CA$ is \emph{totally formal} provided $\CA$ is locally $k$-formal for all $k$.
Thanks to 
\cite[Thm.\ 4.37]{orlikterao:arrangements} and 
Theorem \ref{thm:brandtterao}, if $\CA$ is free, it is totally formal.
Using \cite[Cor.~4.10]{DiPasquale:freeness}, it is easy to 
see that this property of $\CA$ is passed to the restriction $\CA^H$
of $\CA$ to a hyperplane $H$.
For, owing to \cite[Thm.~11]{ziegler:multiarrangements}, the Ziegler multiplicity 
is free on $\CA^H$ and so \cite[Cor.~4.10]{DiPasquale:freeness} shows that 
$\CA^H$ is totally formal.
Note that Corollary \ref{cor:restrition-tot-formal} is valid without the freeness requirement on $\CA$ and gives local formality for any restriction $\CA^H$.

\subsection{Factored arrangements}
\label{ssect:factored}
The notion of a \emph{factored} 
arrangement is due to Terao \cite{terao:factored}.
It generalizes the concept of a supersolvable arrangement, see
\cite[Thm.\ 5.3]{orliksolomonterao:hyperplanes} and 
\cite[Prop.\ 2.67, Thm.\ 3.81]{orlikterao:arrangements}.
factoredness is a combinatorial property
and provides a 
general combinatorial framework to 
deduce tensor factorizations of the underlying Orlik-Solomon algebra,
see also \cite[\S 3.3]{orlikterao:arrangements}.
We recall the relevant notions  
from \cite{terao:factored}
(cf.\  \cite[\S 2.3]{orlikterao:arrangements}):

\begin{defn}
\label{def:factored}
Let $\pi = (\pi_1, \ldots , \pi_s)$ be a partition of $\CA$.
\begin{itemize}
\item[(a)]
$\pi$ is called \emph{independent}, provided 
for any choice $H_i \in \pi_i$ for $1 \le i \le s$,
the resulting $s$ hyperplanes are linearly independent, i.e.\
$r(H_1 \cap \ldots \cap H_s) = s$.
\item[(b)]
Let $X \in L(\CA)$.
The \emph{induced partition} $\pi_X$ of $\CA_X$ is given by the non-empty 
blocks of the form $\pi_i \cap \CA_X$.
\item[(c)]
$\pi$ is a \emph{factorization} of $\CA$  provided 
\begin{itemize}
\item[(i)] $\pi$ is independent, and 
\item[(ii)] for each $X \in L(\CA) \setminus \{V\}$, the induced partition $\pi_X$ admits a block 
which is a singleton. 
\end{itemize}
\end{itemize}
If $\CA$ admits a factorization, then we also say that $\CA$ is \emph{factored}.
\end{defn}

We record some consequences of the main results from \cite{terao:factored} 
(cf.\  \cite[\S 3.3]{orlikterao:arrangements}). Here $A(\CA)$ denotes the \emph{Orlik-Solomon algebra} of $\CA$.

\begin{corollary}
	\label{cor:teraofactored}
	Let  $\pi = (\pi_1, \ldots, \pi_s)$ be a factorization of $\CA$.
	Then the following hold:
	\begin{itemize}
		\item[(i)] $s = r = r(\CA)$ and 
		\[
		\Poin(A(\CA),t) = \prod_{i=1}^r (1 + |\pi_i|t);
		\]
		\item[(ii)]
		the multiset $\{|\pi_1|, \ldots, |\pi_r|\}$ only depends on $\CA$;
		\item[(iii)]
		for any $X \in L(\CA)$, we have
		\[
		r(X) = |\{ i \mid \pi_i \cap \CA_X \ne \varnothing \}|.
		\]  
	\end{itemize}
\end{corollary}

\begin{remark}
\label{rem:factored}
If $\CA$ is non-empty and   
$\pi$ is a factorization of $\CA$, then the non-empty parts of the 
induced partition $\pi_X$ form a factorization of $\CA_X$
for each $X \in L(\CA)\setminus\{V\}$;
cf.~the proof of \cite[Cor.\ 2.11]{terao:factored}.

For  $X$ in $L(\CA)$ modular of rank $r-1$, 
also the converse holds, e.g.~see \cite[Rem.~2.25]{roehrle:ideal}.
\end{remark}

We record the following relation between factored and aspheric arrangements, due to Paris, 
\cite{paris:factored}.

\begin{theorem}
	\label{thm:paris-factored}
	If $\CA$ is a complexified, factored arrangement in $\BBC^3$, then $\CA$ is $K(\pi,1)$.
\end{theorem}

Falk and Randell posed the question whether Theorem \ref{thm:paris-factored} does hold in arbitrary dimensions \cite[Probl.~3.12]{falkrandell:homotopyII}.
Theorem \ref{thm:factored-cformal} is strong evidence for that. 

We close this section with an example illustrating that factoredness is not hereditary, using Theorem \ref{thm:paris-factored}.

\begin{example}
	\label{ex:heredfactored}
Let $\CA$ be the \emph{ideal arrangement} in the Weyl arrangement of type $D_4$ where the associated ideal in the set of positive roots of the root system of type $D_4$  consists of all roots of height at least $4$. Owing to \cite[Thm.~1.31(ii)]{roehrle:ideal}, $\CA$ is (inductively) factored.
In \cite[Ex.~3.2]{amendmoellerroehrle:aspherical}, a restriction $\CA^H$ of $\CA$ is exhibited which is not aspherical. It then follows from Theorem \ref{thm:paris-factored} that $\CA^H$ is not factored.
\end{example}

\section{Proof of Theorem \ref{thm:factored-cformal}}
\label{sec:factored-kformal}

Let $\pi = (\pi_1, \ldots, \pi_\ell)$ be a factorization of $\CA$. We call a subset $\{H_1,\dots,H_\ell\}$ of $\CA$ a \emph{section} of $\pi$ if $H_i\in \pi_i$ for each $i=1,\dots,\ell$. Note that because $\pi$ is a factorization of $\CA$, every section of $\pi$ consists of linearly independent hyperplanes. The following lemma paves the way for our proof of Theorem \ref{thm:factored-cformal}.

\begin{lemma}
	\label{lem:lc-sections}
	Let $\pi$ be a factorization of $\CA$, $S$ be a section of $\pi$ and $H\in \CA$. Then there exists a section $T$ of $\pi$ such that $S\cup\{H\}\subset \lc(T)$.
\end{lemma}

\begin{proof}
	Since $\pi$ is a factorization of $\CA$, 
	if we pick two hyperplanes $H,H'$ in a block of $\pi$,
	then the induced partition $\pi_{H\cap H'}$ contains a singleton block. 
	Define $s(H,H')$ to be the hyperplane in this singleton block. 
	Note that $s(H, H') = H$ if and only if $H=H'$.
	
	Let $\pi = (\pi_1, \ldots, \pi_\ell)$, and let $\iota:\CA\rightarrow\{1,\dots,\ell\}$ be the indicator function of each block, i.e.~$H\in \pi_{\iota(H)}$ for each $H\in \CA$.
	Pick an arbitrary section $S$ of $\pi$
	and an arbitrary hyperplane $H\in \CA$
	and define the sequences
	\begin{align*}
	\HH & :\BBN_0\rightarrow \CA, i \mapsto \HH_i, \\ 	
	\bar{ } & :\BBN_0\rightarrow \{1,\dots,\ell\}, \bar i := \iota(\HH_i) , \text{ and}\\ \SS  & :\BBN_0\rightarrow \{\text{sections of }\pi\}, i \mapsto \SS_i
	\end{align*}
	inductively as follows:
	As starting parameters choose $\HH_0 := H$, and $\SS_0 := S$.
	In the $i$-th step writing $\SS_i=\{K_{i,1},\dots,K_{i,\ell}\}$ such that $K_{i,j}\in \pi_j$ for $j=1,\dots,\ell$, define 
	$$\HH_{i+1} := s(\HH_i, K_{i,\bar i}) \ \text{ and }\ 
	K_{i+1,j} :=
	\begin{cases}
	\HH_i &\text{if}~j=\bar i\\
	K_{i,j} &\text{else}
	\end{cases}
	\quad\text{for }j=1,\dots,\ell.$$
	This means $\HH_i$ is swapped against $K_{i,\bar i}$ 
	in the new section $\SS_{i+1}$ and 
	$\HH_{i+1}$ becomes the singleton hyperplane 
	associated to $\HH_i$ and $K_{i,\bar i}$. Note, both $\HH_i, K_{i,\bar i}$ belong to $\pi_{\bar i}$ and so, by construction, $\SS_{i+1}$ is again a section of $\pi$ for every $i\in \BBN_0$.
	
	The following property is easy to see. 
	For all integers $p,q$ with $0\leq p<q$ 
	and for every $K\in \{\HH_p,K_{p,1},\dots,K_{p,\ell}\}$, 
	either $K\in \SS_q$ 
	or there is a $j\in \{p,\dots,q-1\}$ 
	such that $K=K_{j,\bar j}$. 
	Since $\CA$ is a finite set, we can choose 
	the smallest two integers $p<q$ 
	such that $\HH_p = \HH_q$. 
	We argue by (reverse) induction that 
	$K_{j,\bar j} \in \lc(\SS_q)$ for every $j=0,\dots,q$.
	
	For $j=q$ there is nothing to prove. 
	Assume that $0\leq j<q$ and the hypothesis is true for $k$ with $j<k\leq q$.
	If $\HH_j=K_{j,\bar j}$ or $\HH_{j+1}=K_{j,\bar j}$, 
	then the sequence becomes stationary 
	and we have $K_{j,\bar j}=K_{q,\bar q} \in \SS_q$.

	Otherwise both $\HH_j$ and $\HH_{j+1}$ are either in $\SS_q$ 
	or of the form $K_{k,\bar k}$ for some $j<k\leq q$. 
	Due to the induction hypothesis, we have $\HH_j, \HH_{j+1} \in \lc(\SS_q)$. 
	Since $\HH_{j+1} = s(\HH_j,K_{j,\bar j})$, 
	we have $K_{j,\bar j} \in \CA_{\HH_j\cap \HH_{j+1}}\subset \lc(\SS_q)$,
	which completes the proof.
\end{proof}

\begin{corollary}
	\label{cor:factored-lc}
	If $\pi$ is a factorization of $\CA$,
	then there exists a section $S$ of $\pi$
	such that $\lc(S)=\CA$.
\end{corollary}

\begin{proof}
	Let $S$ be a section of $\pi$ such that
	$\lc(S)$ is of maximal cardinality
	amongst all sections of $\pi$.
	For $H\in \CA$, with Lemma \ref{lem:lc-sections},
	there is a section $T$ such that $\{H\}\cup S\subset \lc(T)$.
	Thus, $\{H\}\cup \lc(S)\subset \lc(T)$,
	and because $S$ was chosen maximal,
	we get $\lc(S)=\lc(T)$ and finally $H\in \lc(S)$.
\end{proof}

Armed with Proposition \ref{prop:lcbasis} and Corollary \ref{cor:factored-lc}, we can now address Theorem \ref{thm:factored-cformal}.

\begin{proof}[Proof of Theorem \ref{thm:factored-cformal}]
Let $\pi$ be a factorization of $\CA$.
With Corollary \ref{cor:factored-lc} there is a section
$S$ of $\pi$ such that $\lc(S)=\CA$.
Since $\pi$ is independent,
we have $r(\cap S)=\ell$.
Thus, by Proposition \ref{prop:lcbasis},
$\CA$ is combinatorially formal.
\end{proof}

We show the versatility of Theorem \ref{thm:factored-cformal}  by improving a result from 
\cite{hogeroehrle:factored}.
In \cite[Thm.~3.5]{hogeroehrle:factored}, the authors prove an addition-deletion theorem for factored arrangements. For, let $\pi$ be a partition of $\CA$ and let
$H_0\in \pi_1$ be fixed. Consider the induced partitions $\pi'$ of $\CA'$ by removing $H_0$ and $\pi''$ of $\CA''$ by intersecting the parts $\pi_2,\dots,\pi_\ell$ with $H_0$. To ensure that $\pi''$ is a factorization, in addition, Hoge and R\"ohrle require the bijectivity of the restriction map $\rho:\CA\setminus\pi_1 \rightarrow \CA''$ defined by $\rho(H)=H\cap H_0$.

\begin{theorem}
	\label{thm:addition-deletion-factored}
	Let $\CA$ be a nonempty arrangement of rank $\ell$ and let $\pi=(\pi_1,\dots,\pi_\ell)$ be a partition of $\CA$. For $H_0\in \pi_1$, consider the arrangement triple $(\CA,\CA',\CA'')$. Two of the following statements imply the third:
	\begin{itemize}
		\item[(i)] $\pi$ is a factorization of $\CA$;
		\item[(ii)] $\pi'$ is a factorization of $\CA'$;
		\item[(iii)] $\rho: \CA\setminus \pi_1\rightarrow \CA''$ is bijective and $\pi''$ is a factorization of $\CA''$.
	\end{itemize}
\end{theorem}

Using Theorem $\ref{thm:factored-cformal}$, we can prove a stronger criterion for the deletion part of the theorem. It turns out that the bijectivity of $\rho:\CA\setminus \pi_1\rightarrow \CA''$ is a strong enough requirement on its own.

\begin{proposition}
	Let $\pi=(\pi_1,\dots,\pi_\ell)$ be a factorization of $\CA$. Let $H_0 \in \pi_1$. Suppose that the restriction map
	$$\rho: \CA\setminus \pi_1 \rightarrow \CA'',\; H\mapsto H\cap H_0$$
	is bijective. Then $\pi'$ is a factorization for $\CA'$ and $\pi''$ is a factorization for $\CA''$.
\end{proposition}
\begin{proof}
	Since $\rho$ is bijective, $\pi''$ is an independent partition of $\CA''$.
	Let $Y\in L(\CA'')$. 
	Note that $Y$ can also be considered as an element of $L(\CA)$. 
	Since $\pi$ is a factorization, there is a singleton part in $\pi_Y$.
	If $\pi_i\cap \CA_Y$ is a singleton for some $i>1$,
	then the bijectivity of $\rho$ implies that $\pi''_i \cap \CA''_Y$ is a singleton as well.
	Therefore, $\pi''$ is not a factorization of $\CA''$ if and only if for some $Y\in L(\CA'')$,
	the only singleton part  in $\pi_Y$ is $\{H_0\}=\pi_1 \cap \CA_Y$.
	Suppose that this is the case.
	Then, if  $H \in \CA_Y\setminus \{H_0\}$, the bijectivity of $\rho$ implies that $\CA_{H\cap H_0} = \{H, H_0\}$,
	 so there are 
	no linear dependencies of rank $2$ in $\CA_Y$ involving $H_0$. 
	With Remark \ref{rem:factored} $\CA_Y$ is factored, 
	so by Theorem \ref{thm:factored-cformal} it is formal. 
	The formality of $\CA_Y$ implies that $H_0$ is a separator inside $\CA_Y$, 
	i.e.~$r(\CA_Y\setminus \{H_0\}) =  r(\CA_Y)-1$. 
	Because of this there exists $X\in L(\CA)$ with $\CA_X = \CA_{Y} \setminus \{H_0\}$. 
	Again, by Remark \ref{rem:factored}, $\CA_X$ is factored, so $\pi_X$ has a singleton part. Our assumption implies $\pi_1\cap \CA_X=\emptyset$ and $|\pi_i \cap \CA_X| = |\pi_i \cap \CA_Y| > 1$ for $2\leq i\leq \ell$,
	a contradiction. 
	It follows that
	$\pi''$ is a factorization for $\CA''$. 
	By the deletion part of Theorem \ref{thm:addition-deletion-factored}, 
	$\pi'$ is also a factorization for $\CA'$.
\end{proof}

\begin{remarks}
	\label{rem:factored-cformal}
	(a). The converse of Theorem \ref{thm:factored-cformal} is false.
	In \cite{moeller:taut}, M\"oller constructs a real $3$-arrangement which is combinatorially formal whose characteristic polynomial does not factor over the integers, thus in particular, it is not factored, cf.~Corollary \ref{cor:teraofactored}(i).
	In addition, M\"oller's example also fails to be  $K(\pi,1)$, as it satisfies the ``simple triangle'' condition of Falk and Randell, \cite[Cor.~3.3]{falkrandell:homotopy}. 
	Note that this example is actually $k$-formal for all $k$. 
	
	(b).
	In view of Theorem \ref{thm:kpi1formal}, 
	Theorem \ref{thm:factored-cformal} is strong support for an affirmative answer to a question raised by Falk and Randell \cite[Probl.~3.12]{falkrandell:homotopyII}, whether 
	complexified factored arrangements are always $K(\pi,1)$, extending Theorem \ref{thm:paris-factored}.
\end{remarks}

\section{Proof of Theorem \ref{thm:restriction-formal}}
\label{sec:restriction-formal}


Let $V$ be an $\ell$-dimensional vector space over some field $\BBK$.
For a subspace $X \subseteq V$ we denote by $\Ann_{V^*}(X) = \{ f \in V^* \mid f|_X \equiv 0\}$
the subspace of $V^*$ of all forms which vanish on $X$.

	Recall from \S \ref{ssect:formality} that for an arrangement $\CA$ in $V$,
	$F(\CA)$ is defined as the kernel of the map
	$\pi_1: \KK(\CA):= \bigoplus_{H \in \CA} \Ann_{V^*}(H) \to V^*, 
	(\alpha_H \mid H \in \CA) \mapsto \sum_{H \in \CA} \alpha_H$.
For a subarrangement $\CB \subseteq \CA$ we clearly have
a natural inclusion $F(\CB) \hookrightarrow F(\CA)$ 
(induced by the natural inclusion $\KK(\CB) \hookrightarrow \KK(\CA)$).
If $\CB = \CA_X$ we denote this inclusion by $i_X$.

Recall, that an arrangement $\CA$ is formal if and only if the map
\begin{align*}
\pi_2 :=\sum_{X \in L_2(\CA)}i_X:\bigoplus_{X \in L_2(\CA)} F(\CA_X) &\to F(\CA), \\
(a_X \mid X \in L_2(\CA)) &\mapsto \sum_{X \in L_2(\CA)} i_X(a_X)
\end{align*}
is surjective.

In the following, we fix an arrangement $\CA$ and an intersection $Z \in L(\CA)$.
The subsequent propositions yield Theorem \ref{thm:restriction-formal} step by step.

\begin{proposition}
	\label{prop:psi1}
	There exists a surjective map $\psi_1: \KK(\CA) \to \KK(\CA^Z)$ such that the diagram
	\begin{equation}
	\label{cd:1}
		\begin{tikzcd}
			\KK(\CA) \ar[r, "\pi_1"] \ar[d, "\psi_1"] & V^* \ar[d, "\res|_{Z^*}"] \\
			\KK(\CA^Z) \ar[r, "\pi_1^Z"] & Z^*
		\end{tikzcd}
	\end{equation}
	commutes, where $\pi_1^Z$ denotes the corresponding map for $\CA^Z$.
	
	Furthermore, for $Y \geq Z$ the induced diagram
	\begin{equation}
	\label{cd:2}
		\begin{tikzcd}
			\KK(\CA_Y) \ar[r, hook] \ar[d, "\psi_{1,Y}"] & \KK(\CA) \ar[d, "\psi_1"] \\
			\KK(\CA_Y^Z) \ar[r, hook] & \KK(\CA^Z)
		\end{tikzcd}
	\end{equation}
	also commutes.
\end{proposition}
\begin{proof}
	For $K \in \CA^Z$ set 
	\begin{align*}
	\psi_K: \bigoplus_{H \in \CA_K \setminus \CA_Z} \Ann_{V^*}(H) &\to \Ann_{Z^*}(K),\\
	(\alpha_H \mid H \in \CA_K \setminus \CA_Z) &\mapsto \sum_{H \in \CA_K \setminus \CA_Z} \alpha_H|_{Z^*},
	\end{align*}
	which is obviously surjective.
	Note that $\CA\setminus  \CA_Z = \bigsqcup_{K \in \CA^Z} (\CA_K\setminus  \CA_Z)$.
	Then	
	\begin{center}
		\begin{tikzcd}
			& \KK(\CA) \ar[dr, "\text{proj}"] \ar[rr, "\psi_1 := (\bigoplus\psi_K) \circ\text{proj}"] & & \KK(\CA^Z)	\\
			& & \KK(\CA\setminus \CA_Z) \ar[ur, "\bigoplus\psi_K"] & 
		\end{tikzcd}
	\end{center}
	does the job (since for $H \in \CA_Z$ and $\alpha_H \in\Ann_{V^*}(H)$ we have $\alpha_H|_{Z^*} = 0$).
	
	The commutativity of the diagram \eqref{cd:2} is obvious. 
\end{proof}

\begin{proposition}
	\label{prop:tpsi1}
	The canonical map $\wt{\psi}_1:F(\CA) \to F(\CA^Z)$ between
	the kernels induced by the diagram  \eqref{cd:1} 
	is surjective.	
\end{proposition}
\begin{proof}
	By Proposition \ref{prop:psi1} we have the following diagram with exact rows:
	\begin{center}
		\begin{tikzcd}
			0 \ar[r] &F(\CA) \ar[d,"\wt{\psi}_1"] \ar[r] & \KK(\CA) \ar[d, "\psi_1"] \ar[r, "\pi_1"] &V^* \ar[d,"\res|_{Z^*}"] \ar[r] &0 \\
			
			0 \ar[r] &F(\CA^Z)  \ar[r] & \KK(\CA^Z) \ar[r, "\pi_1^Z"]  &Z^* \ar[r] &0
		\end{tikzcd}
	\end{center}
	If we look at the map induced by $\pi_1$ between the kernels of $\psi_1$ and $\res|_{Z^*}$,
	we see that under this map $\KK(\CA_Z) \subseteq \ker(\psi_1)$ maps onto $\Ann_{V^*}(Z) = \ker(\res|_{Z^*})$.
	Since $\psi_1$ is surjective, by Proposition \ref{prop:psi1}, the snake lemma yields: $\coker(\wt{\psi}_1) = 0$.	
\end{proof}

\begin{proposition}\label{prop:ImpiZ}
	For the map $\wt{\psi}_1$ we have:
	\[
	\wt{\psi}_1 (\im(\pi_2)) \subseteq \im(\pi_2^Z).
	\]
	In particular, the induced map $\coker(\pi_2) \to \coker(\pi_2^Z)$ is surjective.
\end{proposition}
\begin{proof}
	Since, by definition, we have $\im(\pi_2) =  \sum_{X \in L_2(\CA)} i_X(F(\CA_X))$, it suffices to
	show that $\wt{\psi}_1 (i_X(F(\CA_X))) \subseteq \im(\pi_2^Z)$ for each $X \in L_2(\CA)$.
	
	For $Y \geq X$, we have $i_{X}(F(\CA_X)) \subseteq i_Y(F(\CA_Y))$.	
	If $X \in L_2(\CA)$ then there is a $Y \in L_2(\CA^Z)$ with $Y \geq X$.
	Since the diagram \eqref{cd:2}
	commutes, the same holds for the induced maps on the kernels,
	i.e.~$\wt{\psi}_1 \circ i_Y = i_Y^Z \circ \wt{\psi}_{1,Y}$, 
	and
	we get $$\wt{\psi}_1 (i_X(F(\CA_X))) \subseteq \wt{\psi}_1 (i_Y(F(\CA_Y))) \subseteq i_Y^Z(F(\CA_Y^Z)) \subseteq \im(\pi_2^Z)$$
	which concludes the proof.
\end{proof}

\begin{corollary}[Theorem \ref{thm:restriction-formal}]\label{cor:FormalityRestr}
	Let $\CA$ be a formal arrangement. Then for any $Z \in L(\CA)$ the
	restriction $\CA^Z$ is formal, too.
\end{corollary}


\section{Complements and examples}
\label{s:complements}
In this section, we provide some complements and applications to the main developments.

Thanks to Theorem \ref{thm:restriction-formal}, formality is inherited by restrictions.
In view of the results by DiPasquale \cite{DiPasquale:freeness}
and due to a lack of examples of non-free $k$-formal arrangements in dimensions at least $5$, it is not clear whether Theorem \ref{thm:restriction-formal} extends to higher formality (i.e.\ $k$-formality for $k \geq 3$).
The fact that this is not the case is demonstrated by our next example.
This was found by means of an implementation of a modified
``greedy search'' algorithm based on ideas presented by Cuntz in \cite{Cun21_Greedy}.

\begin{example}
	\label{ex:3formal-not-hereditiary}
	Consider the $5$-arrangement $\CA$ consisting of $11$ hyperplanes in $\BBK^5$ with defining polynomial
	\begin{align*}
	Q(\CA) = \ & x_1x_2x_3x_4x_5(x_1+x_3)(x_1+x_3+x_5)(x_1+x_2+x_4)(x_1+x_2+x_4+x_5) \cdot \\
	& (x_1+x_2+x_3+x_5)(x_1+x_2+x_3+x_4+x_5).
	\end{align*}
  We note that the underlying matroid of $\CA$ is \emph{regular},
  i.e.\ is realizable over any field $\BBK$.
  
  We proceed to show that $\CA$ is $4$-formal (and hence $k$-formal for all $k$).
  To see that $\CA$ is $2$-formal, first observe that there are $6$ relations of rank $2$:
  \begin{align*}
    \alpha_{1}+\alpha_{3}-\alpha_{6}&=0, \\
    \alpha_{2}+\alpha_{7}-\alpha_{10}&=0, \\
    \alpha_{3}+\alpha_{9}-\alpha_{11}&=0, \\
    \alpha_{4}+\alpha_{10}-\alpha_{11}&=0, \\
    \alpha_{5}+\alpha_{6}-\alpha_{7}&=0, \\
    \alpha_{5}+\alpha_{8}-\alpha_{9}&=0, \\
  \end{align*}
  where $\alpha_i$ is the $i$-th linear form in the order as it appears in $Q(\CA)$ above.
  Thus, the map $\pi_2:\bigoplus\limits_{X\in L_2} F_2(\CA_X) \rightarrow F(\CA)$ is defined by the matrix
  \[
    \left(\begin{array}{rrrrrrrrrrr}
    1 & 0 & 1 & 0 & 0 & -1 & 0 & 0 & 0 & 0 & 0 \\
    0 & 1 & 0 & 0 & 0 & 0 & 1 & 0 & 0 & -1 & 0 \\
    0 & 0 & 1 & 0 & 0 & 0 & 0 & 0 & 1 & 0 & -1 \\
    0 & 0 & 0 & 1 & 0 & 0 & 0 & 0 & 0 & 1 & -1 \\
    0 & 0 & 0 & 0 & 1 & 1 & -1 & 0 & 0 & 0 & 0 \\
    0 & 0 & 0 & 0 & 1 & 0 & 0 & 1 & -1 & 0 & 0
    \end{array}\right)\,,
  \]
  which is of rank $6$.
  Since $\dim F(\CA)=|\CA|-r(\CA) = 11-5=6$, it follows that $\pi_2$ is surjective, so $\CA$ is $2$-formal.
  We also see that $\ker\pi_2= R_3(\CA) = \{0\}$ and therefore $R_3(\CA_X)=\{0\}$ for any $X\in L_3(\CA)$ as well.
  Thus, $\pi_3:\bigoplus R_3(\CA_X)\rightarrow R_3(\CA)$ is surjective and its kernel is $R_4(\CA)=\{0\}$, so the map $\pi_4:\bigoplus R_4(\CA_X)\rightarrow R_4(\CA)$ is surjective and $\CA$ is $4$-formal.
 
	The $5$ coordinate hyperplanes $\ker(x_1),\ker(x_2),\ldots,\ker(x_5)$ form an lc-basis
	of $\CA$ and thus the arrangement is combinatorially formal by Proposition \ref{prop:lcbasis}.
		
	Yet, the restriction $\CA^H$ to $H = \ker(x_2)$ in $\BBK^4$,
	which has defining polynomial
	\[
	Q(\CA^H) = x_1x_2x_3x_4(x_1+x_3)(x_1+x_3+x_4)(x_1+x_2)(x_1+x_2+x_4)(x_1+x_2+x_3+x_4),
	\]
	fails to be $3$-formal. To see this, again consider the rank $2$ relations of $\CA^H$:
  \begin{align*}
    \alpha_{1}+\alpha_{2}-\alpha_{7}&=0, \\
    \alpha_{1}+\alpha_{3}-\alpha_{5}&=0, \\
    \alpha_{2}+\alpha_{6}-\alpha_{9}&=0, \\
    \alpha_{3}+\alpha_{8}-\alpha_{9}&=0, \\
    \alpha_{4}+\alpha_{5}-\alpha_{6}&=0, \\
    \alpha_{4}+\alpha_{7}-\alpha_{8}&=0, \\
  \end{align*}
  where $\alpha_i$ is the $i$-th linear form in the order as it appears in $Q(\CA^H)$ above.
  The matrix
  \[
  \left(\begin{array}{rrrrrrrrr}
  1 & 1 & 0 & 0 & 0 & 0 & -1 & 0 & 0 \\
  1 & 0 & 1 & 0 & -1 & 0 & 0 & 0 & 0 \\
  0 & 1 & 0 & 0 & 0 & 1 & 0 & 0 & -1 \\
  0 & 0 & 1 & 0 & 0 & 0 & 0 & 1 & -1 \\
  0 & 0 & 0 & 1 & 1 & -1 & 0 & 0 & 0 \\
  0 & 0 & 0 & 1 & 0 & 0 & 1 & -1 & 0
  \end{array}\right)
  \]
  is of rank $5$, so its kernel $R_3(\CA^H)$ is of dimension $1$. Consider the following table of the nontrivial rank $3$ intersections in $L(\CA^H)$ and the nontrivial rank $2$ flats they are contained in. For the set $\{H_{i_1},\ldots,H_{i_k}\}$ we write $i_1 i_2 \cdots i_k$ as a shorthand.
  \begin{center}
    \tiny
  \begin{tabular}{l|l}
    nontrivial $X\in L_3(\CA^H)$ & nontrivial $Y \in L_2(\CA^H)$ with $Y< X$ \\
    \hline
    $1\, 2\, 3\, 5\, 7$ & $1\, 2\, 5,\ 1\, 3\, 4$ \\
    $1\, 2\, 4\, 7\, 8$ & $1\, 2\, 4,\ 3\, 4\, 5$ \\
    $1\, 2\, 6\, 7\, 9$ & $1\, 2\, 4,\ 2\, 3\, 5$ \\
    $1\, 3\, 4\, 5\, 6$ & $1\, 2\, 4,\ 3\, 4\, 5$ \\
    $1\, 3\, 5\, 8\, 9$ & $1\, 2\, 3,\ 2\, 4\, 5$ \\
    $2\, 3\, 6\, 8\, 9$ & $1\, 3\, 5,\ 2\, 4\, 5$ \\
    $2\, 4\, 5\, 6\, 9$ & $1\, 4\, 5,\ 2\, 3\, 4$ \\
    $3\, 4\, 7\, 8\, 9$ & $1\, 4\, 5,\ 2\, 3\, 4$ \\
    $4\, 5\, 6\, 7\, 8$ & $1\, 2\, 3,\ 1\, 4\, 5$ \\
  \end{tabular}
  \end{center}
	From the table it is obvious that the rank $2$ relations in every rank $3$ localization of $\CA^H$ are linearly independent, so $R_3((\CA^H)_X)=\{0\}$ for every $X\in L_3(\CA^H)$.
  We conclude that the map $\pi_3: \bigoplus R_3((\CA^H)_X) \rightarrow R_3(\CA^H)$ is not surjective and therefore $\CA^H$ is not $3$-formal. 

	Moreover, for $X = \ker(x_1) \cap \ker(x_2)\cap \ker(x_4)$, the rank $3$ localization 	
	$\CA_X$ is obviously  not formal.
	Consequently, $\CA_X$ fails to be factored, $K(\pi,1)$ (for $\BBK = \BBC$), and free, owing to 
	Theorems \ref{thm:factored-cformal}, \ref{thm:kpi1formal}, and \ref{thm:brandtterao},
	respectively, and thus also $\CA$ fails to be factored, $K(\pi,1)$ (for $\BBK = \BBC$), and free, thanks to Remarks \ref{rem:factored}, \ref{rem:local-kpione},  and \ref{rem:local-free}, respectively.
	
	Finally, the argument in the last paragraph also shows that the given lc-basis of $\CA$ does not descend to one for $\CA_X$. 
\end{example}

Owing to \cite{yuzvinsky:obstruction} or by the example above, formality is not inherited by arbitrary localizations.
The following consequence of Theorem \ref{thm:restriction-formal} shows however that it is passed to special localizations.

\begin{corollary}
	\label{cor:modular-formal}
	Suppose $X \in L(\CA)$ is modular.
	If $\CA$ is formal, then so is~$\CA_X$.
\end{corollary}

\begin{proof}
	Since $X \in L(\CA)$ is modular, it follows from Lemma \ref{lem:modular2} that $\CA_X/X \cong \CA^Y/T(\CA)$ for a
	$Y \in L(\CA)$ complementary to $X$.
	As $\CA$ is formal, so is the restriction $\CA^Y$, thanks to Theorem \ref{thm:restriction-formal}.
\end{proof}

The following 
consequence of Lemma \ref{lem:modular1}
gives a converse 
to Corollary \ref{cor:modular-formal} in corank $1$.

\begin{corollary}
	\label{cor:modular-formal-corank1}
	If $X \in L(\CA)$ is modular of corank $1$, then $\CA_X$ is formal if and only if $\CA$ is formal.
\end{corollary}

\bigskip

In the following example, we present a family  of real arrangements $\CK_n$ of rank $n \ge 3$, none of which is free, $K(\pi,1)$, or factored. Nevertheless, each is combinatorially formal. So that in particular, here formality is not a consequence of any of the other sufficient properties. 

\begin{example} 
	\label{ex:nonkpioneII}
	Let $\CB_n$ be the reflection arrangement of the hyperoctahedral group of type
	$B_n$ and let $\CA_{n-1}$ be the braid 
	arrangement of type $A_{n-1}$, a subarrangement of $\CB_n$.
	Let 
	\[
	\CK_n  := \CB_n \setminus \CA_{n-1}
	\]
	be the complement of $\CA_{n-1}$ in $\CB_n$.
	Thus $\CK_n$ has defining polynomial 
	\[
	Q(\CK_n) = \prod\limits_{i = 1}^n x_i\prod\limits_{1 \leq i < j \leq n}\left(x_i + x_j\right).
	\]
	It was shown in \cite[Ex.~2.7]{amendroehrle:ideal} 
	that $\CK_n$ is not $K(\pi,1)$ and not free for $n \ge 3$.
	
	We next show by induction on $n$ that $\CK_n$ is also not factored either for $n \ge 3$. 
	Owing to Theorem \ref{thm:paris-factored},
	$\CK_3$ is not factored, as it is not $K(\pi,1)$.
	Now suppose that $n > 3$ and that $\CK_{n-1}$ is not factored.
	Let $X := \cap_{i = 1}^{n-1} \ker x_i$. Then one readily checks 
	that 
	\[
	(\CK_n)_X \cong \CK_{n-1}. 
	\]
	It follows from our induction hypothesis and Remark \ref{rem:factored} that 
	also $\CK_n$ fails to be factored.
	
	Finally, we observe that $\CK_n$ admits an lc-basis. For, let $\CB \subset \CK_n$ be the Boolean subarrangement, i.e.~$Q(\CB) = \prod_{i = 1}^n x_i$. Then one easily checks that $\lc(\CB) = \CK_n$. 
	It thus follows from Proposition \ref{prop:lcbasis} that $\CK_n$ is combinatorially formal; formality for $\CK_3$ was already observed by Falk and Randell in  \cite[(3.12)]{falkrandell:homotopy}.
\end{example}

\bigskip {\bf Acknowledgments}: 
The second author is grateful to the Max Planck Institute for Mathematics in Bonn for support
during his stay as a postdoctoral fellow.


\bigskip

\bibliographystyle{amsalpha}

\newcommand{\etalchar}[1]{$^{#1}$}
\providecommand{\bysame}{\leavevmode\hbox to3em{\hrulefill}\thinspace}
\providecommand{\MR}{\relax\ifhmode\unskip\space\fi MR }
\providecommand{\MRhref}[2]{%
  \href{http://www.ams.org/mathscinet-getitem?mr=#1}{#2} }
\providecommand{\href}[2]{#2}


\end{document}